\documentclass[11pt,letterpaper]{article}
\usepackage{amsfonts, amsmath, amssymb, amscd, amsthm, color, graphicx, mathrsfs, wasysym, setspace, mdwlist}
\newcommand{\Wr }{\,{\rm Wr}\,}
\newcommand{\e}{\varepsilon}
\newcommand{\comp}{{\rm comp}}
\renewcommand{\d}{{\rm d}}

\newcommand{\wre}{\textrm{\,wr\,}}
\newcommand{\Wre}{\textrm{\,Wr\,}}
\newcommand{\diam}{{\rm diam}}
\newtheorem{thm}{Theorem}[section]
\newtheorem{cor}[thm]{Corollary}
\newtheorem{lem}[thm]{Lemma}
\newtheorem{prop}[thm]{Proposition}
\newtheorem{prob}[thm]{Problem}

\theoremstyle{definition}
\newtheorem{defn}[thm]{Definition}
\theoremstyle{remark}
\newtheorem{rem}[thm]{Remark}
\newtheorem{ex}[thm]{Example}

\renewcommand{\c}{{\rm c}}

\newcommand{\com}{{\rm comp}}
\newfont{\eufm}{eufm10}

\begin{document}

\title{A quasi-isometric embedding theorem for groups}

\author{A.Yu.  Olshanskii, D.V. Osin}

\date{}

\maketitle

\begin{abstract}
We show that every group $H$ of at most exponential growth with respect to some left invariant metric admits a bi-Lipschitz embedding into a finitely generated group $G$ such that $G$ is amenable (respectively, solvable, satisfies a non-trivial identity,  elementary amenable, of finite decomposition complexity, etc.) whenever $H$ is. We also discuss some applications to compression functions of Lipschitz embeddings into uniformly convex Banach spaces, F\o lner functions, and elementary classes of amenable groups.
\end{abstract}

\section{Introduction}

It is well known that every countable group can be embedded in a group generated by $2$ elements. This theorem was first proved by  Higman, B.H. Neumann, and H. Neumann in 1949 \cite{HNN} using HNN-extensions. Since then many alternative constructions have been found and the result has been strengthened in various ways. Most of the subsequent improvements were motivated by the desire to better control either the algebraic structure of the resulting finitely generated group or geometric properties of the embedding.

The original proof of the Higman-Neumann-Neumann theorem leads to ``large" finitely generated groups even if one starts with a relatively ``small" countable group; for instance, the resulting finitely generated group always contains non-abelian free subgroups. In the paper \cite{NN}, B.H. Neumann and H. Neumann suggested an alternative approach based on wreath products, which allowed them to show that every countable solvable group can be embedded in a $2$-generated solvable group. This approach was further developed by P. Hall \cite{Hall} and used by Phillips \cite{P} and  Wilson \cite{W} to prove analogous embedding theorems for torsion and residually finite groups.

In another direction, the Higman-Neumann-Neumann theorem was strengthened by the first author in \cite{Ols}. Recall that a map $\ell\colon H\to \mathbb N\cup\{ 0\}$ is a {\it length function} on a group $H$ if it satisfies the following conditions.
\begin{enumerate}
\item[($\bf L_1$)] $\ell (h)=0$ iff $h=1$.
\item[($\bf L_2$)] $\ell (h)=\ell (h^{-1})$ for any $h\in H$.
\item[($\bf L_3$)] $\ell (gh)\le \ell (g)+\ell(h)$ for any $g,h\in H$.
\end{enumerate}
Further we say that the \emph{growth of $H$ with respect to $\ell $ is at most exponential}  if there is a constant $a$ such that for every $n\in \mathbb N$, we have
\begin{equation}\label{ex}
   \#\{ h\in H\mid \ell (h)\le n\} \le a^n
\end{equation}
(this property was added to (L1)-(L3)  in \cite{Ols}). A length function $\ell\colon H\to \mathbb N\cup\{ 0\}$ defines a left-invariant metric on $H$ by $\d (g, h) = \ell(g^{-1}h)$ and vice versa.

If $G$ is a group generated by a finite set $X$ and $H$ is a (not necessary finitely generated) subgroup of $G$, the restriction of the word length $|\cdot |_X$ to $H$ obviously defines a length function on $H$ and the growth of $H$ with respect to $|\cdot |_X$ is at most exponential. In \cite{Ols}, the first author proved that any length function $\ell $ on a group $H$ satisfying (\ref{ex}) can be realized up to bi-Lipschitz equivalence by  such an embedding. The method of \cite{Ols} is based on small cancellation techniques and consequently suffers from the same problem as  the original Higman-Neumann-Neumann embedding: even if one starts with a ``small" (say, abelian) group $H$, the resulting group $G$ contains many free subgroups.

The goal of this paper is to suggest yet another construction which allows us to control both the geometry of the embedding and the algebraic structure of the resulting finitely generated group. Given a group $H$, we denote by $\mathcal E(H)$ the class of all groups $K$ such that every finitely generated subgroup of $K$ embeds in a finite direct power of $H$.  Our main result is the following.

\begin{thm}\label{main}
Let $H$ be a group of at most exponential growth with respect to a length function $\ell\colon H\to \mathbb N\cup\{ 0\}$. Then $H$ embeds (as a subgroup) into a group $G$ generated by a finite set $X$ such that the following conditions hold.
\begin{enumerate}
\item[(a)] There exists $c>0$ such that for every $h\in H$, we have $c|h|_X\le  \ell(h)\le |h|_X$.

\item[(b)] $G$ has a normal series $G_1\lhd G_2 \lhd G$, where
$G_1$ is abelian and intersects $H$ trivially,  $G_2/G_1\in {\cal E}(H)$, and $G/G_2$ is solvable of derived length at most $3$.
\end{enumerate}
\end{thm}

In particular, our embedding allows us to carry over a wide range of properties from the group $H$ to the group $G$. For instance, we have the following.
\begin{cor}\label{cor0}
In the notation of Theorem \ref{main}, if $H$ is solvable (respectively, satisfies a non-trivial identity, elementary amenable, amenable, has property $A$, has finite decomposition complexity, uniformly embeds in a Hilbert space, etc.), then so is $G$.
\end{cor}
Recall that property $A$ was introduced by  Yu \cite{Yu} and groups of finite decomposition complexity were introduced by Guentner, Tessera, and Yu \cite{GTY} with motivation coming from the Novikov conjecture and topological rigidity of manifolds, respectively. For definitions, properties, and applications we refer to \cite{GTY,RTY,Yu} and references therein. Groups uniformly embeddable in Hilbert spaces are discussed below. 

The classes of elementary amenable groups, amenable groups, countable groups with property A, countable groups of finite decomposition complexity, and countable groups uniformly embeddable in a Hilbert space contain abelian groups and are closed with respect to the operations of taking direct unions and subgroups. In addition, the first four classes are closed under taking arbitrary extensions; the  class of countable groups uniformly embeddable in a Hilbert space is closed under taking direct products and  extensions by abelian (or, more generally, exact) groups, see \cite{Cho,DG,GTY,vN}. These properties and Theorem 1.1 obviously imply Corollary 1.2

For amenable and elementary amenable groups, even the following fact was unknown. This is remarked by Gromov in \cite[Section 9.3]{Gro08}, where the reader can also find some potential applications of the existence of such an embedding.

\begin{cor}\label{cor1}
Every countable elementary amenable (respectively, amenable) group can be embedded in a finitely generated elementary amenable (respectively, amenable) group.
\end{cor}

The proof of Theorem \ref{main} is based on a modified version of the construction of P. Hall \cite{Hall}, which in turn goes back to B.H. Neumann and H. Neumann \cite{NN}. In general, this construction does not preserve elementary amenability, amenability, etc., as it uses unrestricted wreath products. A small improvement based on the existence of parallelogram-free subsets (see Definition \ref{para}) in finitely generated solvable groups fixes this problem. However the original construction does not allow one to control the distortion of the embedding, so an essential modification is necessary to ensure (a). Our main technical tool here is a ``metric" version of the Magnus embedding, which seems to be of independent interest (see Section 2).

One possible application of Theorem \ref{main} is to constructing solvable, amenable, etc., groups with unusual geometric properties. In general, it is much easier to build an interesting geometry inside an infinitely generated group; then Theorem \ref{main} allows us to embed it in a finitely generated one. This philosophy has a few almost immediate implementations. We briefly discuss them below and refer to Section 3 for definitions and details.

Recall that to each map from a finitely generated group $G$ to a metric space $(S, \d _S)$, one associates a non-decreasing \emph{compression function} $\comp_f\colon \mathbb R_+\to \mathbb R_+$ defined by
$$
\comp _f (x) =\inf_{\d _X(u,v)\ge x} \d_S (f(u), f(v)),
$$
where $\d_X$ is the word metric on $G$ with respect to a finite generating set $X$.

Given two functions $r,s\colon \mathbb R_+\to \mathbb R_+$, we write $r\preceq s$ if there exists $C>0$ such that $$r(x)\le C\cdot s(Cx)+C$$ for every $x\in \mathbb R_+$. As usual, $r\sim s$ if  $r\preceq s$ and $s\preceq r$. Up to this equivalence, the compression function $\comp _f$ is independent of the choice of a particular finite generating set of $G$. If $f$ is Lipschitz and satisfies $\comp _f(x)\to \infty $ as $x\to \infty$, then $f$ is called a \emph{uniform embedding.}

The study of group embeddings into Hilbert (or, more generally, Banach) spaces was initiated by Gromov in \cite{Gro}. Motivated by his ideas, Yu \cite{Yu} and later Kasparov and Yu  \cite{KY} proved that finitely generated groups uniformly embeddable in a Hilbert space (respectively, a uniformly convex Banach space) satisfy the coarse Novikov conjecture. Another interesting result  was proved by  Guentner and Kaminker in \cite{GK}. They showed that if a finitely generated group $G$ admits a uniform embedding in a Hilbert space with compression $\succeq x^\e$ for some $\e >1/2$, then the reduced $C^*$-algebra of $G$ is exact; moreover, if the embedding is equivariant, then $G$ is amenable. On the other hand, by a result of Brown and Guentner \cite{BG}, any metric space of bounded geometry can be uniformly embedded into the $\ell^2$-sum $\bigoplus l^{p_n} (\mathbb N )$ for some sequence of numbers $p_n \in (1,+\infty )$, $p_n \to \infty $.

In \cite{ADS}, Arzhantseva, Drutu, and Sapir showed that for every function $\rho\colon\mathbb R\to \mathbb R$ satisfying $\lim\limits_{x\to \infty} \rho (x)=\infty $, there exists a finitely generated group $G$ such that every Lipschitz map from $G$ to a uniformly convex Banach space has compression function $\preceq \rho $. The groups constructed in \cite{ADS} contain free subgroups and hence are not amenable. On the other hand, computations made in various papers (see, e.g., \cite{AGS,T} and references therein) suggest that for amenable groups the situation can be different.

For some time it was unknown even whether every finitely generated amenable group admits a Lipschitz embedding in a Hilbert space with compression function $\succeq x^\e $ for some $\e>0$. This question was asked in \cite{ADS,AGS,T} and answered negatively by Austin in \cite{A}. Austin also remarks that his approach can probably be extended to give a finitely generated amenable group $G$ such that every Lipschitz map $f\colon G\to L_p$, $p\in [1, \infty)$, has compression function $\preceq \log x$. However his methods do not seem to work beyond this limit and he asks whether every finitely generated amenable group $G$ admits a Lipschitz embeddings $f\colon G \to L_p$ for every $p\in  [1,\infty)$ with compression function $\succeq \log x$. We show that the answer is negative and, moreover, an analogue of the Arzhantseva-Drutu-Sapir result holds for amenable groups.

\begin{cor}\label{cor2}
Let $ \rho \colon \mathbb R_+\to \mathbb R_+$ be any function such that $\lim\limits_{x\to \infty}\rho (x)=\infty $. Then there exists a finitely generated elementary amenable group $G$ such that for every Lipschitz embedding $f$ of $G$ into a uniformly convex Banach space, the compression function of $f$ satisfies $\comp_f  \preceq \rho $.
\end{cor}

The proof of the corollary is inspired by \cite{ADS} and uses expanders constructed by Lafforgue \cite{Laf}.

Our approach can also be used to obtain some known results about F\o lner functions introduced by Vershik \cite{Ver}. For a finitely generated amenable group $G$, its F\o lner function, $F\o l_G\colon \mathbb N\to \mathbb N$  measures the asymptotic growth of F\o lner sets of $G$. Vershik conjectured that there exist amenable groups $G$ with the F\o lner functions growing arbitrary fast and his conjecture remained open for several decades until Erschler proved it in \cite{Ers}. The groups constructed in \cite{Ers} are of intermediate growth and, in particular, they are amenable but not elementary amenable \cite{Cho}.

Note that for many elementary amenable groups $G$, $F\o l_G$ is bounded from above by an iterated exponential function. For instance, this is true for all finitely generated solvable groups by the main result of \cite{Ers03}. On the other hand, the following was announced by Eschler in \cite{Ers03} and later proved by Gromov in \cite{Gro08}. Corollary \ref{cor2} allows us to recover this result.

\begin{cor}[Erschler-Gromov]\label{cor3}
For any function $\sigma \colon \mathbb N\to \mathbb N$, there exists an elementary amenable group $G$ with $F\o l _G \succeq \sigma $.
\end{cor}

 The last application discussed in this paper concerns the class $EA$ of elementary amenable groups. Recall that every $G\in EA$ can be ``constructed" from finite and abelian groups by taking subgroups, quotients, extensions, and direct unions. The \emph{elementary class} of $G$, $\c (G)$, is an ordinal number that measures the complexity of this procedure (see Section 3 for the precise definition). It is easy to see that for every countable group $G\in EA$, one has $\c (G)< \omega_1$, where $\omega _1$ is the first uncountable ordinal; however it is unclear how large $\c(G)$ can be. For finitely generated groups this question is mentioned by Gromov in \cite[Section 9.3]{Gro08}. For subgroups of Thompson's group $F$ it was also addressed by Brin in \cite{Brin}, where he showed that for every non-limit ordinal $\alpha \le \omega^2+1$ there exists an elementary amenable subgroup $G\le F$ such that $\c(G)=\alpha $. However Brin remarks that his approach does not work beyond $\omega^2 +1$, at least for groups of orientation preserving piecewise linear self homeomorphisms of the unit interval.

In the corollary below we give a complete description of ordinals that can be realized as elementary classes of countable elementary amenable groups.

\begin{cor} \label{EG} Let $EA_c$ and $EA_{fg}$ denote the sets of all countable elementary amenable and finitely generated elementary amenable groups, respectively.  Then we have
$$\c(EA_c)=\{ \alpha +1\mid \alpha <\omega _1\} \cup\{ 0\} $$
and
$$\c(EA_{fg})=\{ \alpha +2\mid \alpha <\omega _1\} \cup\{ 0, 1\}.$$
\end{cor}

Yet another application of our main result will be given in the forthcoming joint paper \cite{DO}
by Tara Davis and the first author: \emph{Every super-additive
function $\mathbb N\cup\{ 0\}\to \mathbb N \cup \{ 0\}$ can be realized (up to equivalence) as the distortion function of a cyclic subgroup in a finitely generated solvable group. }(For definitions and background see Gromov's paper \cite{Gro}).

We conclude with the following.

\begin{prob}
Does every recursively presented countable amenable (or elementary amenable) group embed into a finitely presented amenable group?
\end{prob}

\section{Preliminary results}

Let $G$ be a group generated by a set $X\subseteq G$. Given an element $g\in G$, one defines the word length of $g$ with respect to $X$, $|g|_X$, as the length of a shortest word in the alphabet $X\cup X^{-1}$ that represents  $g$ in $G$.

We recall the definition of the (standard) wreath product of two groups
$A$ and $B.$ The {\it base subgroup} $\overline W$ is the set of functions from $B$ to $A$ with
pointwise multiplication. The group $B$ acts on $\overline W$ from the left by automorphisms,
such that for $f\in W$
and $b\in B,$  the function  $b\circ f $ is given by

$$(b\circ f)(x) =  f(x b) \;\;\; for \;\;\; every\;\;\;  x\in B$$

This action defines a semidirect product $\overline WB$ called the {\it Cartesian (or unrestricted) wreath product}
of the groups $A$ and $B,$ denoted $A \Wre B.$ Hence we have $bfb^{-1} = b\circ f$ in this group,
and every element of $A\Wre B$ can be uniquely written as $fb,$ where $ f\in \overline W, b\in B$.

The functions $B \to A$ with finite support, i.e. with the condition $f(x)=1$ for almost every $x\in B,$
form a subgroup $W$ in $\overline W$. Respectively, we have the subgroup $ WB$ of $A \Wre B,$
denoted $A \wre B$ and called the {\it direct (or restricted) wreath product} of the groups $A$ and $B.$
$W$ is called \emph{the base} of  $A \wre B$ and is the direct product of the subgroups $A(b)$ ($b\in B$) isomorphic to $A,$ such that and  $A(b)=bA(1)b^{-1}$ in $A \wre B.$ One may identify the subgroup $A(1)$ with $A,$ and so the
wreath product is generated by the subgroups $A$ and $B.$

If the group $A$ is abelian, then so is $W,$ and one may use the additive notation for the
elements of the base subgroup $W$ of $A \wre B.$ The base $W$ becomes a module over the group
ring $\mathbb Z B.$ In particular, if $A$ is a free abelian group with basis $(a_1,\dots, a_n),$
then $W$ is a free $\mathbb Z B$-module, and in the module notation, every element of
$W$ has a unique presentation as $t_1 \circ a_1 +\dots + t_n \circ a_n$ with $t_1,\dots,t_n\in \mathbb Z B.$

\paragraph{Generalized Magnus embedding. }
Let $F$ be a free group (not necessary finitely generated), $N$ a normal subgroup of $F$. The standard Magnus homomorphism may be thought of as an embedding of a group of the form $F/[N,N]$ in the wreath product $A \wre F/N $, where $A\simeq F/[F,F]$. The goal of this section is to describe a modified version that also depends on a length function $\ell $ on $F/N$. Our main result in this direction is Lemma \ref{Magnus}; its analogue for the standard version of the Magnus embedding, Corollary \ref{Mcor}, also seems new.

We start by recalling the standard definition. Let $H$ be a group with a set of generators $(h_i)_{i\in I}$, $F$ the free group with basis $(x_i)_{i\in I}$.  Denote by $N$ the kernel of the homomorphism $F\to H$ given by $x_i\mapsto h_i$. Let also $A$ be the free abelian group with basis $(a_i)_{i\in I}$ and let $V= A \wre H$. The standard \emph{Magnus homomorphism} $\mu_{0}\colon F\to V$ maps $x_i$ to $ a_i h_i$ ($i\in I$). By Magnus's theorem \cite{M}, $\ker \mu_{0} = [N,N]$, the derived subgroup of $N$.

Every element of $V$ has a unique representation of the form $wg,$ where $g\in H$ and $w$ belongs to the base subgroup $W,$  and so  $w=\sum_{i\in I}  t_i \circ a_i,$ where $t_i\in \mathbb Z H$ and $t_i=0$ for all but finitely many $i$.
By the Remeslennikov -- Sokolov criterion, if $wg\in \mu_{0}(F),$ then $$\sum_{i\in I} t_i(h_i-1) = g-1$$ in $\mathbb Z H.$ (We use only the easier half of `iff' from \cite{RS}, and for $wg=\mu_{0}(u),$ this equation is directly verified
by the induction on the length of $u.$)

Assume now that  $H$ is a group with a length function $\ell(h)$ and a set of generators $(h_i)_{i\in I}$.  We set $l_i=\max(\ell(h_i), 1)$ and
define a \emph{generalized Magnus homomorphism} from the free group $F$ to $V$ by the formula
$$
\mu (x_i)= a_i^{l_i}h_i, \;\; i\in I.
$$
One can think of the standard Magnus homomorphism as a special case of the generalized one corresponding to the length function $\ell\colon \mathbb N\cup\{ 0\}$ that takes values $1$ on all nontrivial elements of $H$.

Thus $\mu$ is a homomorphism from $F$
to the subgroup isomorphic to the wreath product $A' \wre H,$ where $A'$ is the free abelian subgroup of $A$ with basis $(l_i a_i)_{i\in I}$ (in additive notation). Therefore the Remeslennikov -- Sokolov
property looks now as follows.
Let $w=\sum_{i\in I}  t_i\circ a_i,$ where $t_i\in \mathbb Z H,$ and let $g\in H.$
\bigskip

\noindent(RS)\;\; {\it If $wg\in \mu(F),$ then $t_i = l_is_i$ for some $s_i\in \mathbb Z H$ and  $\sum_{i\in I}  s_i(h_i-1) = g-1$ in $\mathbb Z H.$}

\bigskip

For an element $t=\sum_{h\in H} k_h h$ of $\mathbb Z H$ ($k_h\in \mathbb Z$), we define its \emph{norm} by
$$||t||= \sum_{h\in H} |k_h|.$$ Further for $w=\sum_{i\in I} t_i \circ  a_i\in W,$ we set  $$||w||= \sum_{i\in I} ||t_i||.$$

\begin{lem}\label{Magnus}
If in the above notation $wg\in \mu(F),$ then $\ell(g)\le ||w||.$
\end{lem}

\proof We may assume that $g\ne 1$ since otherwise there is nothing to prove.

Every element $s_i$ is a sum of the form $\sum_j \pm g_{ij},$ where $g_{ij}\in H,$
and in this sum, an element of $H$ can occur $|k_{ij}|$ times either with sign $+$ or with sign $-$ (but not with both signs). Therefore  equality (RS) can be rewritten as

\begin{equation} \label{star}
\sum_i(\sum_j \pm g_{ij}) (h_i-1) = g-1
\end{equation}

Note that the right-hand side of (\ref{star}) has only two terms.  We  define a labeled \emph{cancellation graph} $\Gamma$ reflecting the cancellations in the left-hand side.
By definition, $\Gamma$ has vertices $(i,j,1)$ labeled by $\mp g_{ij}$ and the vertices $(i,j,2)$ labeled by $\pm g_{ij}h_i$.

Every vertex $(i,j,1)$ is connected with  $(i,j,2)$ by a red edge.
So every vertex is connected by a red edge with exactly one other vertex.
 Since almost all terms in the left-hand side of (\ref{star}) cancel
out, there is a  pairing on the set of vertices without two vertices $o$ and $o^\prime$ labeled by $-1$ and $g$, respectively, such that one of the vertices in each pair is labeled
by some $x\in H$ and the other is labeled by $-x$. We fix the pairing and connect two vertices of a pair by a blue edge.

\begin{figure}
 \centering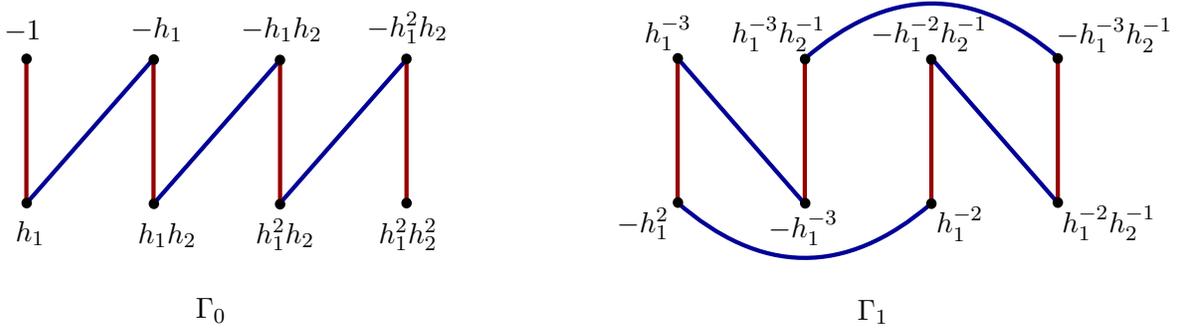\\
  \caption{The cancellation graph for the image of the word $x_1^{-3}x_2^{-1}x_1x_2x_1^3x_2x_1x_2$ under the standard Magnus homomorphism.}\label{fig1}
\end{figure}

As an example, consider the free abelian group $H$ with basis $\{ h_1,h_2\}$. For simplicity take the standard word length $\ell$ on $H$ with respect to the generators $\{ h_1,h_2\}$. In this case our generalized Magnus homomorphism $\mu \colon F\to A\wre H$, where $F=F(x_1, x_2)$ is the free group with basis $\{ x_1, x_2\} $ and $A$ is the free abelian group with basis $\{ a_1,a_2\}$, coincides  with the standard Magnus homomorphism defined by $\mu (x_i)=a_ih_i$ for $i=1,2$. Let $$f=x_1^{-3}x_2^{-1}x_1x_2x_1^3x_2x_1x_2$$ Then it is straightforward to compute
$$
\mu (f)= \big((h_1h_2+h_1^{-3}h_2^{-1} - h_1^{-3} +1)\circ a_1  + (h_1^2h_2 + h_1 +h_1^{-2} h_2^{-1} - h_1^{-3}h_2^{-1})\circ a_2\big) h_1^2h_2^2.
$$
Then (\ref{star}) takes the form
$$
\big(h_1h_2+h_1^{-3}h_2^{-1} - h_1^{-3} +1\big)(h_1-1)  +  \big(h_1^2h_2 + h_1 +h_1^{-2} h_2^{-1} - h_1^{-3}h_2^{-1}\big) (h_2-1) = h_1^2h_2^2  -  1.
$$
The corresponding cancellation graph is drawn on Fig. \ref{fig1}.

Obviously every vertex of $\Gamma$ belongs to exactly one red edge, and every vertex except for $o, o^\prime,$ belongs to exactly one blue edge, i.e., $o$ and $ o^\prime$ have degree $1$ while other vertices have degree $2.$
It follows that  $\Gamma$ decomposes into connected components $\Gamma_0,$ $\Gamma_1, \dots,$
where $\Gamma_0$ is a simple arc connecting $o$ and $o^\prime$ and other components are simple loops.
The red edges and the blue ones must alternate in the directed path $o-o^\prime=p=e_0f_1e_1\dots, f_d e_d.$
Moreover $e_0$ and $e_d$ (and so all edges $e_i$) are red since no blue edges start/end in $o$ or in $o^\prime.$

Note that $||s_i||$ is equal to the number of vertices of the form $(i, j,1)$ (with various indices~$j$) in $\Gamma,$ and $||w||=\sum_i ||t_i||=\sum_i l_i||s_i||.$ Therefore to obtain $||w||,$
we may assign the weight $l_i$ to every vertex $(i,j,1)$ and sum these weights. Hence to bound
$||w||$ {\it below}, it suffices to sum the assigned weights of the vertices only along the path $p$.
We have $||w||\ge \sum_k l_{i_k}$, over all vertices of the form $(i_k,j_k,1)$ passed by $p$. Since every
red edge $e$ of $p$ connects some $(i,j,1)$ and $(i,j,2),$ we can assign the weight $l_i$ to such a red edge $e$.
So
\begin{equation}\label{2stars}
 ||w|| \ge \sum_{k=0}^d l_{i_k},
 \end{equation}
where $l_{i_k}=\max(\ell(h_{i_k}), 1)$ is the weight of $e_k.$

Let $-z_{2k}$ and $z_{2k+1}$ be the labels of the initial and the terminal vertices  of the red edge $e_k$ in $p.$
Using the definition of red and blue edges,  we see for $k>0,$ that  $z_{2k+1} = -z_{2k}h_{i_k}^{\pm 1} = z_{2k-1}h_{i_k}^{\pm 1}.$ It follows that
 $\ell(z_{2k+1})\le \ell(z_{2k-1})+ \ell(h_{i_k}).$ Likewise $\ell(z_1)\le \ell(z_0)+\ell(h_{i_0}).$
Hence by induction, $$\ell(g)=\ell(z_{2d+1})\le \ell(z_0)+\sum_{k=0}^d \ell(h_{i_k})\le 0+\sum_{k=0}^d l_{i_k}$$
since $\ell(z_0)=\ell(1)=0.$ This inequality and (\ref{2stars}) prove the lemma.
\endproof

\begin{cor}\label{Mcor}
Let $H$ be a group with a (finite or infinite) set of generators $X=(h_i)_{i\in I}$
and let $\mu$ be the Magnus homomorphism $x_i\mapsto a_ih_i$ ($i\in I$) from the free group $F=F(X)$
to the wreath product $V=A \wre H$ of a free abelian group $A$ with basis $(a_i)_{i\in I}$ and
$H.$ If $\mu(f)=wg,$ where $f\in F$, $g\in H,$ and $w$ belongs to the base subgroup of  $V,$ then

(a) $|g|_X\le ||w||;$

(b) $|g|_X = ||w||$ if and only if the canonical images of $f$ in $H\simeq F/N$ and in $F/[N,N]$ have equal lengths with respect to
$X$; here $N$ is the kernel of the homomorphism $\epsilon\colon x_i\mapsto h_i$.
\end{cor}

\proof (a) The function $\ell(h) = |h|_X$ is a length function, and $\max(\ell(h_i),1)=1$ for every
$h_i\in X.$ So the generalized Magnus homomorphism $\mu$ from Lemma \ref{Magnus} is just the standard Magnus
homomorphism in this case. Hence the statement follows from Lemma \ref{Magnus}.

(b) Let $r$ be the length of $f$ in $F/[N,N],$ that is, by Magnus' theorem, the length of $\mu(f)$ in $\mu(F)$ with
respect to $(a_ih_i)_{i\in I}$.
Thus $\mu(f)= y_1\dots y_r$, where each $y_j$ is of the form $(a_ih_i)^{\pm 1}.$
The element $g$ is the $\epsilon$-image of $f$ in $H$.

Assume first that $r = |g|_X.$ We want to prove the equality  $|g|_X = ||w||,$ and by statement (a), it
suffices to prove the inequality $||w||\le r$. So it sufficient to show by induction that if $y_1\dots y_{j-1}=w'g'$
with $g'\in H$, $w'\in V$, and $||w'||\le j-1,$ then $y_1\dots y_{j-1}y_j=w''g''$ with $||w''||\le j$.
But $w'g'(a_jh_j)^{\pm 1} = w' k a_j^{\pm 1}k^{-1} (g'(h_j)^{\pm 1}),$ where $k=g'$ or $k=g'h_j^{-1}.$
Therefore $||w''||\le ||w'||+ ||k a_j^{\pm 1}k^{-1}||\le j-1 +1 =j$, as required.

Now we assume that $|g|_X = ||w||\ge 1$ and will use the notation and arguments of Lemma \ref{Magnus} (but with $l_i=1$ and $s_i=t_i$).
The number of red edges in the component $\Gamma_0$ is at most $r=|g|_X$ since $||w||$ is the number of
red edges in the entire $\Gamma$. On the other hand, $r$ is at least the number of red edges in $\Gamma_0$
since the difference of the lengths of $g_{ij}$ and $g_{ij}h_i$ labeling (with signs) the ends of
a red edge differ by at most $1,$ and such a difference is $0$ for blue edges of the path $p$ connecting
the vertices labeled by $1$ and $g.$ Hence $r=d+1$ and passing  every red edge $e_j$ directed in $p$ from $1$ to
$g$ we increase the length of the vertex label exactly by $1$. (Moreover, the graph $\Gamma$ is connected.) Let the vertices of $e_d$ be labeled by
$\mp g_{ij}$ and by $\pm g_{ij}h_i$ for some $i,j$, where $\pm g_{ij}$ is a summand of $s_i$. So there are two cases: either (1) $g=g_{ij}$ and $|g_{ij}h_i|_X=r-1$ or
(2) $g= g_{ij}h_i$ and $|g_{ij}|_X=r-1$.

{\bf Case (1)} Let $g'= g_{ij}h_i$ and $w'g'= (wg)(a_ih_i),$ i.e., $w'=w(ga_ig^{-1})$.
We can say using module notation that passing from $w$ to $w'$ we add the summand $g=g_{ij}$ to
$s_i$. Since it cancels with  the summand $-g_{ij}$ of $s_i,$ we have $||w'||=||w||-1=r-1.$
Thus $||w'||=|g'|_X=r-1,$ and by induction, the length of $w'g'$ in the generators
$(a_kh_k)_{k\in I}$ is equal to $r-1.$ So the length of $wg$ in these generators is $\le r-1+1$, as required.

{\bf Case (2)} Let $g'= g_{ij}$ and $w'g'= (wg)(a_ih_i)^{-1}.$ Now $w'= w(g_{ij}a_i^{-1}g_{ij}^{-1}),$ i.e.,
we add $-g_{ij}$ to cancel it with the summand $g_{ij}$ of $s_i$ and complete the proof as in Case 1.
\endproof

\paragraph{Parallelogram-free subsets in metabelian groups.}
The following concept is important for describing the algebraic structure of the finitely generated group from Theorem \ref{main}.

\begin{defn}\label{parafree}
Let $u_1, u_2, u_3, u_4$ be elements of a group $G$ such that $u_i\ne u_{i+1},$ for arbitrary
$i$ taken modulo $4.$ We say that the configuration $(u_1,u_2,u_3,u_4)$ is a {\it parallelogram}
if
\begin{equation}\label{para}
u_1u_2^{-1}u_3u_4^{-1}=1
\end{equation}
in $G$. A subset $P$ of a group $G$ is  {\emph{parallelogram-free}} if it contains no parallelograms.
\end{defn}

\begin{lem}\label{sh} Let $G$ be a group, $P$ a subset of $G$.
\begin{enumerate}
\item[(1)] For every $g\in G$, if $P$ is parallelogram-free so are the sets  $gP$ and $Pg$.

\item[(2)] The following properties are equivalent:

\begin{enumerate}
\item[(a)] $P$ is parallelogram-free.

\item[(b)] $\#\{P\cap gP\}\le 1$ for every nontrivial $g\in G$.

\item[(c)] $\#\{P\cap Pg\}\le 1$ for every nontrivial $g\in G$.
\end{enumerate}
\end{enumerate}
\end{lem}

\proof Claim (1) is true since the condition (\ref{para}) is invariant under simultaneous  multiplication of all the elements $u_1,\dots, u_4$  by $g$ from the left or from the right.

To prove (2) we note that  if $(u_1,u_2,u_3,u_4)$ is a parallelogram in $P$, then two distinct elements $u_1$ and $u_2$ belong to both $P$ and $gP$ for non-trivial $g=u_1u_4^{-1}$ since $gu_4=u_1$ and $gu_3=u_1u_4^{-1}u_3=u_2$ by (\ref{para}). Therefore (b) implies (a).
Conversely, (a) implies (b) since given two distinct elements $u_1$ and $u_2$ in  $P\cap gP,$
we have the parallelogram $(u_1, g^{-1}u_1, g^{-1}u_2,u_2)$ in $P$.
Similarly we obtain the equivalence of (a) and (c).
\endproof

\begin{defn}\label{nice1} Recall that a subset $P$ of a group $M$ with a finite set of generators $S$
has exponential growth if
 there exist constants $c>1$ and $\lambda>0$ such that for every $n\in \mathbb N$,
$$\# \{ w\in P\mid \lambda |w|_S < n+1\}\ge c^n$$
\end{defn}

\begin{lem}\label{S}
There exists a  finitely generated metabelian group $M$ with a parallelogram-free subset $P$
of exponential growth.
\end{lem}

\begin{proof}
Take $M=\mathbb Z \wre \mathbb Z$ and $S=\{x_0,y_0\}$, where $x_0$ and $y_0$ generate the wreathed infinite cyclic groups. It is well known and easy to prove that the group $M$ has exponential growth: every ball $B_r$ of radius $r\ge 0$ centered at $1$ has at least $k^r$ elements for some $k>1.$ We will construct $P$ as $P=\bigcup_{r=0}^{\infty}P_r$, where $P_0=\{1\}$ and for $r>0$, $P_r$ is a maximal  parallelogram-free extension of $P_{r-1}$ in $B_r$.
To prove the lemma, it suffices to show that $\#P_r> \frac12 k^{r/3}$ for every $r\ge 0.$
In turn, it will be sufficient to prove that if $N$ is a parallelogram-free subset of $B_r$
and $n=\# N  \le \frac12 k^{r/3},$ then for some $x\in B_r\backslash N$, the subset
$N'=N\cup\{x\}$ is also parallelogram-free.


We will require that $x$ satisfies none of the  equations (\ref{e1}--\ref{e3}) below.
\begin{equation}\label{e1}
 x = u, \;\; {\rm where} \;\; u\in N.
 \end{equation}
   The total number of solutions of all such equations is $n$.

\begin{equation}\label{e2}
  xu^{-1}vw^{-1}=1, \;\; {\rm where} \;\;u,v,w\in N.
\end{equation}
The total number of solutions is at most $n^3$.
\begin{equation}\label{e3}
xu^{-1}xv^{-1}=1,\;\; {\rm where} \;\;u,v\in N.
\end{equation}
Equation (\ref{e3}) is equivalent to
the equation $(xu^{-1})^2= vu^{-1}.$ To prove that the total number of solutions of all such equations is at most $n^2,$ it suffices to prove that any equation of the form  $y^2=a$ has at most one
solution in $M$.  Indeed, let $(w_1g)^2 = (w_2h)^2$ where $g,h$ belong to the active infinite cyclic group
and $w_1, w_2$ belong to the base subgroup of the wreath product $M.$ Then we immediately have $g^2=h^2,$
and so $g=h.$ Now we obtain  $(1+g) \circ w_1= (1+g) \circ w_2$ in the module notation. But the free module over the
group ring of an infinite cyclic group has no module torsion, whence $w_1=w_2.$

Now since $$k^r > \frac12 k^{r/3} + \frac14 k^{2r/3} + \frac{1}{8} k^{r}\ge n+n^2+n^3,$$ there is $x\in B_r$ such that
$x$ does not satisfy any of the equations (\ref{e1}--\ref{e3}). Note  that the
set $N'=N\cup\{x\}$ is bigger than $N$ since $x$ satisfies no equation of the form (\ref{e1}).
The set $N'$ is parallelogram-free because $N$ is such, and $x$ is not a solution of any of the equations
(\ref{e2}-\ref{e3}).
This completes the proof of the lemma.
\end{proof}

\section{Proof of the main theorem}

\paragraph{The construction.}
Throughout this section we use the notation $[x,y]=xyx^{-1}y^{-1}$. We start by describing our main construction, which consists of three steps, and then prove the theorem. The notation introduced below remains valid until the end of this section.

Let again $H$ be a group with a length function $\ell\colon H\to \mathbb N\cup \{0\}$. Let $A$ be the free abelian group with basis
$\{a_h\mid h\in H\backslash\{1\}\}$ and let $$Y = \{ a_h,\,  a_h^{\ell(h)}h \mid h\in H\backslash\{1\}\} \subseteq A \wre H.$$  Obviously $Y$ generates the wreath product
\begin{equation}\label{defV}
V=A \wre H.
\end{equation}

Further consider the unrestricted wreath product $V \Wre \mathbb Z$, where $V$ is defined by (\ref{defV}). Let $U$ be the union of sets $U_1$ and $U_2$ defined as follows:
$U_1$ consists of all functions $f_{1,h}$ from the base of $V \Wr \mathbb Z$
 such that
$$
f_{1,h}(n)=\left\{\begin{array}{cl}
                1, & {\rm if \;} n\le 0 \\
                a_h^{l(h)}h, & {\rm if \;} n> 0
              \end{array}\right.
$$
and $U_2$ consists of all functions $f_{2,h}$ such that
$$
f_{2,h}(n)=\left\{\begin{array}{cl}
                1, & {\rm if \;} n\le 0 \\
                a_h, & {\rm if \;} n> 0 .
              \end{array}\right.
$$

Let $t$ be a generator of $\mathbb Z$. For definiteness let $t=1$.
Let $$K\le V \Wr \mathbb Z$$ be the subgroup generated by $Z=U\cup\{t\}$.

Note that $[t, f_{1,h}]=tf_{1,h}t^{-1} f_{1,h}^{-1}$, considered as a function $\mathbb Z\to V$, takes only one nontrivial value $a_h^{\ell(h)}h$ at $0$. Similarly, $[t, f_{2,h}]$ takes only one nontrivial
value $a_h$ at $0.$ In what follows, these commutators
will be identified with the corresponding elements of $Y$. Thus we have a sequence of embeddings  $$H\le V\le K.$$

Finally let $M$ be a metabelian group with a subset $P$ provided by Lemma \ref{S}, $S$ a finite generating set of $M$.  By Lemma \ref{sh} (1), we may shift $P$ and assume that $1\in P.$ Decreasing the constant $\lambda$ in Definition \ref{nice1} if necessary, we can assume that $c\ge \max\{ a, 3\}$, where $a$ and $c$ are the constants from (\ref{ex}) and Definition \ref{nice1}, respectively.
 Thus there exists a subset $P_0\subseteq P$ such that $1\in P_0$ and for every $n\ge 0,$
$$
\# \{ w\in P_0\mid \lambda |w|_S < n+1\} =2\# \{ h\in H\backslash\{1\}\mid \ell(h)\le n\}+1 .
$$
We list all elements of $P_0=\{ 1, w_1, w_2, \ldots \}$ and $U\cup\{1\}=\{1, u_1,u_2, \ldots \}$ in such a way that
\begin{equation}\label{hiwi}
\lambda |w_i|_S< \ell(h)+1
\end{equation}
if $u_i= a_{h}^{\ell(h)}h$ or $u_i= a_{h}.$
Let $B$ denote the base of the wreath product $K\, {\rm Wr}\, M$. Let $g\colon M\to K$ be the element of $B$ such that
$$
g(x)=
\left\{\begin{array}{cl}
         t, & {\rm if \; } x=1, \\
         u_i, & {\rm if \; } x=w_i \; {\rm for \; some\;} i\in \mathbb N, \\
         1, & {\rm if \; } x\notin P_0.
       \end{array}\right.
$$
Let $$G\le K\, {\rm Wr}\, M$$ be the subgroup generated by the finite set $X=S\cup\{ g\}$.

\paragraph{The proof.}
To prove Theorem \ref{main} we will need few preliminary results. The first one is a corollary of Lemma \ref{Magnus}.

\begin{cor} \label{l}   For every non-trivial element $h\in H\le V,$ we have $|h|_Y=\ell(h)+1.$
\end{cor}

 \proof Note that $h= a_h^{-\ell(h)}( a_h^{\ell(h)}h),$ and therefore $|h|_Y\le 1+\ell(h).$

Assume now that $h=y_1\dots y_r,$ where $r=|h|_Y$ and $y_j\in Y^{\pm 1}$ for $j=1,\dots,r.$
Let us move all $y_j$'s of the form $( a_h^{\ell(h)}h)^{\pm 1}$ to the right using conjugation.
We obtain $h= uv,$ where $v\in \mu(F)$ for the homomorphism $\mu\colon F\to V$  given by the rule $x_h\mapsto  a_h^{\ell(h)}h,$ and $u= y^\prime_1\dots y^\prime_q$ is the product of some conjugates
of the factors $y_j$ having the form $a_h^{\pm 1}.$ (We assume that the number of such factors in the factorization of $h$ is $q$.) It follows that $||y^\prime_j||=1$ since the norm does not change under conjugation.

As in Lemma \ref{Magnus},  we have $v=wg$ in $V$, where $g=h$ (use the projection $V\to H$).
We can apply the assertion of Lemma \ref{Magnus} since $\max(\ell(h),1)=\ell(h)$ for $h\in H\backslash\{1\},$ and conclude
that $\ell(h) \le ||w||.$

On the one hand, $h=uwg,$ and so $uw=1$ and  $||v||=||w||.$
On the other hand, $||v||\le q$ since $||y^\prime_j||=1$ for every $y^\prime_j.$ Hence
$q\ge ||w||.$
But since $h\ne 1$, we must also  have at least one factor of the form  $(a_h^{\ell(h)}h)^{\pm 1}$
in the product $y_1\dots y_r$ . Therefore, $$|h|_Y=r\ge q+1\ge ||w||+1\ge \ell(h)+1$$ as required.
\endproof

\begin{lem} \label{VZ}
For any $h\in H$, we have $
\ell(h)\le |h|_Z$.
\end{lem}
\begin{proof}
Let $L$ denote the base of the wreath product $V \Wre \mathbb Z$ and let $\pi \colon L\to V$ be the projection which to every function $f\colon \mathbb Z\to V$ assigns its value $f(0)$. Fix $h\in H$. Let $w$ be a word in the alphabet $Z^{\pm 1}$ representing $h$ and such that $|h|_Z=|w|$. Applying the standard rewriting process to $w$ we obtain the equality \begin{equation}\label{h}
h=(t^{\alpha_1}f_1^{\pm 1} t^{-\alpha_1})\cdots (t^{\alpha_m}f_{m}^{\pm 1}t^{-\alpha_m}),
\end{equation}
where $f_1,\dots, f_m\in U$   and $m\le |h|_Z.$
Now applying $\pi $ to both sides of (\ref{h}) we obtain $h=y_1^{\pm 1} \cdots y_m^{\pm 1}$, where $y_i\in Y\cup \{1\}$ because every value of a function from $U$ (and of a conjugate
of it by a power of $t$) is either trivial
or belongs to $Y.$ This and Corollary  \ref{l} imply the inequalities $l(h)\le |h|_Y\le m \le |h|_Z.$
\end{proof}

\begin{lem}\label{HtoG}
\begin{enumerate}
\item[(a)] For any $r\in M$, the support of the function  $[g,r g r^{-1}]\colon M\to K$ consists of at
most one element and its value at this element is a commutator in $K$.

\item[(b)] The  group $G$ contains all functions $M\to V$ mapping nontrivial elements to $1.$
\end{enumerate}
 \end{lem}

\begin{proof} (a) We may assume that $r\ne 1.$ Let $u\in M.$ Then $$[g,r g r^{-1}](u)= [g(u),(rgr^{-1})(u)]= [g(u),g(ur)]$$
Assume that this commutator is not trivial. Then both $g(u)$ and $g(ur)$ are nontrivial, and so both $u$ and $ur$ belong to $P_0\subset P$ by the definition of $g.$ If $[g,r g r^{-1}](u')=1$ for $u'\ne u,$
then $u'$ and $u'r$ are also elements of $P,$ and so $\#(P\cap  Pr)\ge 2$. This contradicts Lemma \ref{sh} (2).
Thus part (a) is proved.

(b) Recall that every generator $y\in Y\backslash\{1\}$ of $V$ is identified with a commutator $[t,u_i]$ in $K$ for some $u_i\in U.$
Therefore
\begin{equation}\label{gw}
[g,w_i g w_i^{-1}](1)= [g(1),(w_igw_i^{-1})(1)]= [g(1),g(w_i)]= [t,u_i]=y
\end{equation}
while
$[g,w_igw_i^{-1}](x)=1$
whenever $x\ne 1$ by part (a).
\end{proof}

By Lemma \ref{HtoG}, the group $V$ (and therefore $H$) can be regarded as a subgroup of $G.$

\begin{lem} \label{a} There is a  positive constant $\theta$  such that $\theta|h|_X\le \ell(h)\le |h|_X$
for every element $h$ of $H.$
\end{lem}

\begin{proof} We may assume that $h\ne 1.$ Recall that $h=a_h^{-l(h)} (a_h^{l(h)}h)=(y^\prime)^{-\ell(h)}y$ for some $y, y'\in Y.$ Thus $h$ is a product of $1+l(h)$ generators from $Y^{\pm 1}.$ In turn (see Lemma
\ref{HtoG} (a) and
(\ref{gw})), $y=[g,w_i g w_i^{-1}]$ for some $i$, and hence $y$ has length $\le 4+4|w_i|_S$ with respect to $X=S\cup\{ g\}$. Similarly, $y^\prime= [g,w_{i^\prime} g w_{i^\prime}^{-1}]$ for some $i^\prime$, but the same
computation as in Lemma \ref{HtoG}, shows that
$$
(y^\prime)^{-\ell(h)}= [g,w_{i^\prime} g^{-\ell(h)} w_{i^\prime}^{-1}]
$$
in $G,$ and so the length of $(y^\prime)^{-l(h)}$
with respect to $X$ does not exceed $2\ell(h)+4|w_{i^\prime}|_S+2.$  Since by the choice of
$w_i$ and $w_{i^\prime},$ their $S$-lengths (and $X$-lengths) do not exceed $\lambda^{-1}(\ell(h)+1)$ (see (\ref{hiwi}), we have
$$|h|_X \le (4+ 4\lambda^{-1}(\ell(h)+1))+ (2\ell(h)+4\lambda^{-1}(\ell(h)+1) +2) \le (8+16\lambda^{-1}) \ell(h).$$ So the first inequality is
proved with $\theta=  (8+16\lambda^{-1})^{-1}.$

To prove the second inequality, we note that every $h\in H$, considered as an element of $G$, can be written as $$h=p_0g^{\alpha _1} p_1\cdots g^{\alpha _k} p_k$$ for some $p_0, \ldots , p_k\in M$ and some integers $\alpha _1, \ldots, \alpha _k$ such that
\begin{equation}\label{hx}
|h|_X=\sum\limits_{j=0}^{k} |p_k|_S +\sum\limits_{i=1}^k |\alpha _j|.
\end{equation}
Since $h$ belongs to the base $B$, we have $p_0\cdots p_k=1$. Thus we can rewrite $h$ in the form
\begin{equation}\label{h1}
h=r_1g^{\alpha_1}r_1^{-1} \cdots r_kg^{\alpha_k}r_{k}^{-1},
\end{equation}
where $r_1=p_0$, $r_2=p_0p_1$, $\ldots $, $r_{k}=p_0p_1\cdots p_{k-1}=p_k^{-1}$. Applying, to both sides of (\ref{h1}), the projection $B\to K$ that maps each $b\colon M\to K$ to $b(1)$, we obtain the equality
$$
h=x_1^{\alpha_1} \cdots x_k^{\alpha_k}
$$
in the group $K$, where $x_j=(r_jgr_j^{-1})(1)=g(r_j)$. Note that $x_j$ is equal to either $1$ or $t\in Z$, or some $u\in U\subset Z.$ Therefore $|h|_Z\le \sum_j |\alpha_j|\le |h|_X$ by (\ref{hx}). Finally, by Lemma \ref{VZ}, $l(h)\le |h|_Z\le |h|_X.$
\end{proof}

Given a group $H$, we denote by ${\bf D}H$ the class of finite direct powers of $H.$
Further let ${\bf S} {\bf D} H$ be the class of all subgroups of groups from ${\bf D} H$. Finally let ${\cal E}(H)={\bf L}{\bf S} {\bf D}H$ be the class of all groups which are locally in ${\bf S} {\bf D} H$,  i.e., $G\in {\cal E}(H)$ if every finitely generated subgroup of $G$ belongs to  ${\bf S} {\bf D}H$.

\begin{lem} \label{E} Let $T$ be a set, $H$ a group, and $(f_i)_{i\in I}$ an arbitrary set of functions $T \to H,$ where every
$f_i$ has a finite range in the group $H.$ Then  the subgroup $\langle f_i\mid i\in I\rangle $ of the Cartesian power $H^T$ belongs to the class $ {\cal E}(H).$
\end{lem}
\proof It suffices to proof that for every finite subset $J\subset I$,  $\langle f_i\mid i\in J \rangle \in {\bf S} {\bf D}H.$

Since every $f_i$ has a finite range and $J$ is also finite, there is a finite partition $ T=\bigcup_{k=1}^r T_k$
such that the restriction of $f_i$ to $T_k$ is a constant for every  $i\in J$ and every $k\le r.$ It follows
that every function $f$ of the subgroup $G_J=\langle f_i\mid i\in J \rangle $ is a constant on every particular $T_k.$ Therefore
the mapping $f\mapsto (f(T_1),\dots, f(T_r))$ is an injective homomorphism from $G_J$ to a direct product
of $r$ copies of $H.$ \endproof

\begin{proof}[Proof of Theorem \ref{main}.] We have constructed the embeddings $H\hookrightarrow V \hookrightarrow G.$ By Lemma \ref{a}, the embedding of $H$ in $G$ satisfies condition (a)
of Theorem \ref{main} with $c=\theta$. It remains to prove (b).

Note that if $A$ is a normal abelian subgroup of a group $C$, then all functions with
values in $A$ form a normal abelian subgroup  in an arbitrary wreath product $C \Wr D.$ Using this
observation and the constructions of $V$, $K$, and $G$ as subgroups of the corresponding wreath  products, we see
that the free abelian subgroup $A=\langle a_h\mid h\in H\backslash\{1\}\rangle $ is contained in a normal
abelian subgroup $G_1$ of $G$ such that the canonical image $\bar V$ of $V$ in $\bar G=G/G_1$ is equal to $\bar H$ (the image of $H$ in $\bar G$) and so it is isomorphic to
$H$. Further to obtain the image $\bar G$ one has to replace  each element $a_h$ by the identity
in the definitions of the functions from the set $U$ and, respectively,  in the definition of the function $g$. (So the image $\bar g$ of $g$ in $\bar G$ takes values in $\bar K$ which is isomorphic to a subgroup of $\bar V \Wre \mathbb Z\simeq H \Wre \mathbb Z.$)

Let $R/G_1$ be the normal closure of $\bar g$ in $\bar G$. Obviously $  G/ R\simeq M$,
the metabelian group. So for $G_2=[R,R]G_1,$ we have that the quotient group $G/G_2$ is
solvable of derived length $\le 3.$

The group $\bar R=R/G_1$ is generated by all conjugates $r\bar gr^{-1},$ where $r\in M$. Therefore
the group $\bar G_2=G_2/G_1$ is generated
by all the elements $d[\bar g, r\bar g r^{-1}]d^{-1}$, where $r\in M$ and $d\in \bar G$.
It follows from Lemma \ref{HtoG} (a) that  $[\bar g, r\bar g r^{-1}]$, and so $d[\bar g, r^{-1}\bar gr]d^{-1}$,  considered as a function $M\to \bar K$, is either trivial or takes a nontrivial value at exactly one element of $M$ and this value is an element of $[\bar K,\bar K]$.
Now by Lemma \ref{E}, we have $\bar G_2\in {\cal E}([\bar K,\bar K])$.

The group $\bar K$ is generated by $\bar t$ and the functions $\bar f_{h}=\bar f_{1,h}$ taking values in $\bar H\simeq H.$
Therefore the group $[\bar K,\bar K]$ is generated by all $d[\bar t,\bar f_h]d^{-1}$ and $d[\bar f_h,\bar f_{h^\prime}]d^{-1},$ where $d\in \bar K.$ Each of the functions $f_h$  has finite range in $H.$ So do
the commutators $[\bar t,\bar f_h]$ and $[\bar f_h,\bar f_{h^\prime}].$  The same property holds for
their conjugates by any element $d$ since $d$ is a finite product, where every factor is either
$\bar t^{\pm 1}$ or $\bar f_g^{\pm 1}$ (with finite range) for some $g\in H$. Therefore $ [\bar K,\bar K]\in {\cal E}(H)$ by Lemma \ref{E}. Since $G_2/G_1=\bar G_2 \in {\cal E}([\bar K,\bar K])$, we have $G_2/G_1\in {\cal E}(H)$, and part (b) of the theorem is proved.
\end{proof}

\section{Applications}

\paragraph{Compression functions of Lipschitz embeddings in uniformly convex  Banach spaces.}
Recall that the {\em compression function} $\com (f) \colon\mathbb R_+\to \mathbb R_+$ of  a map $f$ from a metric space $(X,\d_X)$ to a metric space $(Y,\d_Y)$ $f$  is defined by
$$
\com_f (x) =\inf\limits_{\d_X(u,v)\ge x} \d_Y(f(u), f(v)).
$$

We start by summarizing some essential features of Lafforgue's construction of expanders \cite{Laf}.

\begin{lem}\label{ADS}
There exists an infinite group $\Gamma $ with a finite generating set $X$ and a sequence of finite index normal subgroups $$\Gamma\rhd N_1\rhd N_2\rhd \ldots $$ with trivial intersection such that the following holds. Let $(G_k, \d_k)$ denote the quotient group $\Gamma /N_k$ endowed with the word metric corresponding to the image of the generating set $X$. Let $E$ be a uniformly convex Banach space with a norm $\| \cdot \| $. Then there exist constants $R, \kappa >0$ such that for every $1$-Lipschitz embedding $f\colon (G_k, \d_k)\to (E, \| \cdot\| )$ one can find two elements  $u,v\in G_k$ satisfying $$\| f(u) - f(v)\| \le R$$ and $$\d _k(u,v)\ge \kappa \diam (G_k),$$ where $\diam (G_k)$ is the diameter of $G_k$ with respect to $\d_k$.
\end{lem}

In fact, one can take $\Gamma$ to be any co-compact lattice in $SL_3(F)$ for a non-archimedean local field $F$ \cite{Laf}. The proof of the lemma can be found in Section 3 of \cite{ADS} (see Corollary 3.5 there) and is quite elementary modulo Lafforgue's paper.

We will also need the following property.

\begin{lem}\label{exp}
Let $G$ be group generated by a finite set $X$. Let $G_1, G_2, \ldots $ be a sequence of quotients of $G$. Denote by $\d_k$ the word metric on $G_k$ with respect to the image of the set $X$. Let $H=\prod _{k=1}^\infty G_k$ be the direct product of $G_k$'s. For an element $h=(g_k)\in \prod_{k=1}^\infty  G_k$ we define
\begin{equation}\label{met}
\ell (h)=\sum\limits_{k=1}^\infty k\d_k(1, g_k).
\end{equation}
Then $\ell$ is a length function on $H$ and the growth of $H$ with respect to $\ell $ is at most exponential.
\end{lem}

\begin{proof}
The fact that $\ell $ is a length function is obvious. To show that the growth of $H$ with respect to $\ell $ is at most exponential it suffices to deal with the case when $G_1\cong G_2\cong \ldots \cong G$. In the latter case, the map $H\to G \wre \mathbb Z$ that sends $G_k$ to $t^kGt^{-k}$, where $t$ is a generator of $\mathbb Z$, extends to a Lipschitz embedding of $H$ in  $G \wre \mathbb Z$. Indeed the length of the image of $h=(g_k)$ in $G \wre \mathbb Z$ does not exceed $\sum (2k+\d(1,g_k))$
over all $g_k\ne 1$, where $\d$ is the word metric on $G$ with respect to $x$. This in turn does not exceed  $\sum\limits_{k=1}^\infty 3k\d(1, g_k)=3\ell(h).$  Since the wreath product is finitely generated,  the claim is proved.
\end{proof}

\begin{proof}[Proof of Corollary \ref{cor2}]
Fix any function $ \rho \colon \mathbb R_+\to \mathbb R_+$ such that $\lim\limits_{x\to \infty}\rho (x)=\infty $. It suffices to prove the corollary for the function $\rho^\prime (x)=\inf\limits_{t\in [x, \infty)} \rho (t)$. Hence we can assume that $\rho$ is non-decreasing without loss of generality. Further let $\mathcal G= \{(G_k, \d _k)\}$ be the family of finite groups provided by Lemma \ref{ADS} and let $D_k=\diam (G_k)$ denote the diameter of $G_k$ with respect to $\d_k$. Since $D_k\to \infty $ as $k\to \infty $, passing to a subsequence of $\mathcal G$ if necessary we can assume that
\begin{equation}\label{subseq}
\rho (D_k) \ge k
\end{equation}
for all $k\in \mathbb N$.

Let $H=\prod_{k=1}^\infty  G_k$ be the direct product of all $G_k$ and let $\ell \colon H\to \mathbb N\cup\{0\}$ be the function on $H$ defined by (\ref{met}). Then, by Lemma \ref{exp}, $\ell$ is a length function on $H$ and the growth of $H$ with respect to $\ell $ is at most exponential. Note that $H$ is elementary amenable. Hence by Theorem \ref{main}, there exists an elementary amenable group $G$ containing $H$ and generated by a finite set $X$ and a constant $c>0$ such that for every $h\in H$, we have
\begin{equation}\label{hch}
c|h|_X\le \ell (h)\le |h|_X .
\end{equation}

Let $f$ be a Lipschitz embedding of $G$ in a uniformly convex Banach space $E$, $L$ its Lipschitz constant. Note that $\comp_{\frac1L f}\sim \comp_f$. Thus we can assume that $f$ is $1$-Lipschitz. By (\ref{hch}) and (\ref{met}), for every $u,v\in G_k\le G$ we have
$$
\| f(u)-f(v)\| \le \d_X (u,v) \le  \ell (u^{-1}v)/c = k \d_k (u,v)/c.
$$
Thus the embedding $G_k\to G$ composed with $f$ gives us a $(k/c)$-Lipschitz map $(G_k, d_k)\to E$. Rescaling and applying Lemma \ref{ADS}, we can find two elements $u,v\in G_k$ such that
\begin{equation}\label{comp1}
\| f(u) - f(v)\| \le kR/c
\end{equation}
and
\begin{equation}\label{comp2}
\d _X(u,v)\ge \ell (u^{-1}v)\ge \d _k(u,v)\ge \kappa D_k.
\end{equation}
By the definition of the compression function, the inequalities (\ref{comp1}) and (\ref{comp2}) imply
$$
\comp_f (\kappa D_{k})  \le kR/c.
$$
Finally, for any large enough $x\in \mathbb R_+$, we have $\kappa D_{k-1} \le x\le \kappa D_{k}$ for some $k\ge 2$. Using (\ref{subseq}) and the fact that $\rho $ is non-decreasing, we obtain
$$
\comp_f(x)\le \comp_f (\kappa D_{k})  \le \frac{kR}{c} \le \frac{2(k-1)R}{c}\le \frac{2R \rho (D_{k-1})}c\le \frac{2R}c \rho \left(\frac{x}{\kappa}\right).
$$
Thus $\comp _f \preceq \rho .$
\end{proof}

\paragraph{F\o lner functions.}
We recall the definition of a F\o lner function of an amenable group introduced by Vershik \cite{Ver}. A finite subset $A$ of a finitely generated group $G$ is $\e $-F\o lner (with respect to a fixed finite generating set $X$ of $G$) if
$$
\sum\limits_{x\in X} |Ax \triangle A| \le \e |A|,
$$
where $\triangle$ denotes symmetric difference. The F\o lner function of an amenable group $G$ (with
respect to $X$) is defined by
$$
F\o l_{X,G}(n) = \min\{|A| : A \subseteq G\; {\rm is\; a \;
{\it 1/n}-F\o lner\; w.r.t.\;} G\} .
$$
Up to the standard equivalence relation induced by $\preceq$, $F\o l_{X,G}$ is independent of the choice of a finite generating set $X$. The corresponding equivalence class is denoted by $F\o l_{G}$.

Let $\mathcal C$ be a class of groups. By a \emph{rank function} $\rho $ on $\mathcal C$ we mean a map $\rho \colon \mathcal C \to P$, where $P$ is a poset, such that the following conditions hold.
\begin{enumerate}
\item[($\bf R_1$)] For every $G\in \mathcal C$ and every $H\le G$, if $H\in \mathcal C$ then $\rho (H)\le \rho (G)$.
\item[($\bf R_2$)] For every $G\in \mathcal C$ and every quotient group $Q$ of $G$, if $Q\in \mathcal C$ then $\rho (Q)\le \rho (G)$.
\end{enumerate}
A rank function $\rho\colon \mathcal C\to P$ is {\it unbounded} if $\rho (\mathcal C)$ has no largest element.

\begin{ex}
\begin{enumerate}
\item[(a)] The derived length is an unbounded rank function on the class of solvable groups.
\item[(b)] Similarly the function $\c$ defined below is a rank function on the class of elementary amenable groups which takes values in ordinal numbers. Indeed the conditions ($\bf R_1$) and ($\bf R_2$) follow from Lemma \ref{Chou}. If we restrict our attention to the set of countable elementary amenable groups, then Corollary \ref{EG} shows that $\c $ is unbounded.
\end{enumerate}
\end{ex}

Recall that a group $G$ is \emph{$SQ$-universal } if every countable group can be embedded into a quotient of $G$. The Higman-Neumann-Neumann theorem discussed in the introduction may be restated as follows: the free group of rank $2$ is $SQ$-universal. Given a class of groups $\mathcal C$, we say that a group $G\in \mathcal C$ is \emph{$SQ$-universal group in the class $\mathcal C$} if every countable group from $C$ embeds in a quotient of $G$.

\begin{lem}\label{obvious}
Let $\mathcal C$ be a class of countable groups with an unbounded rank function.
\begin{enumerate}
\item[(a)] There is no $SQ$-universal group in the class $\mathcal C$.

\item[(b)] If, in addition, every countable family of groups from $\mathcal C$ embeds simultaneously into a group from $\mathcal C$, then for any $\sigma \in \rho (\mathcal C)$, there exists an uncountable chain in $\rho (\mathcal C)$ with minimal element $\sigma$.
\end{enumerate}
\end{lem}

\begin{proof}
The first statement is obvious from the definition. To prove (b) we first note the following.

\medskip

\noindent ($\ast$)\;\;\; Every countable subset of $\rho(C)$ has an upper bound.

\medskip

\noindent Indeed given a countable subset $c\in \rho (C)$, consider the groups $G_1, G_2, \ldots $ from $\mathcal C$ such that $\rho (\{ G_1, G_2, \ldots \})=c$. Let $G\in \mathcal C$ be a group which contains all $G_i$'s. Then $\rho (G)$ is an upper bound for $c$.

Suppose now that every chain in $\rho (\mathcal C)$ with minimal element $\sigma$ is countable. Then ($\ast$) and the Zorn Lemma imply that the set $\{\tau\in \rho (\mathcal C)\mid \tau \succeq \sigma \}$ contains a maximal element $\mu$. Using ($\ast$) again, we conclude that $\mu $ the largest element in $\rho (\mathcal C)$, which contradicts our assumption.
\end{proof}

Let $F\o l$ denote the function which maps a finitely generated amenable group $G$ to $F\o l_G$.  Note that the relation $\preceq$ defined in the introduction is a preorder, and hence induces a partial order on the set of corresponding equivalence classes. We denote this order also by $\preceq$. It was proved by Erschler in \cite{Ers03} that with respect to $\preceq$, $F\o l$ is a rank function on the class of finitely generated amenable groups. In another paper \cite{Ers}, she also showed that $F\o l$ is unbounded on the class of finitely generated amenable groups. Moreover, is is unbounded on the class of finitely generated groups of subexponential growth. As a consequence of Corollary \ref{cor2}, we obtain below that $F\o l$ is unbounded on the class of elementary amenable groups.

\begin{proof}[Proof of Corollary \ref{cor3}]
There is a standard construction of a uniform embedding of amenable groups into a Hilbert space, where the compression of the embedding is controlled by the F\o lner functions (see, e.g., \cite{Yu}). That is, if $F\o l$ was bounded on the class of finitely generated elementary amenable groups, there would exist Lipschitz embeddings of every finitely generated elementary amenable group into a Hilbert space with compression function $\succeq \rho$ for some $\rho $ satisfying the conditions of Corollary \ref{cor2}. A contradiction.
\end{proof}

\begin{rem}
Lemma \ref{obvious} also implies that for every $\sigma \colon \mathbb N\to \mathbb N$, there is an uncountable chain of F\o lner functions of elementary amenable groups bounded by $\sigma$ from below.
As a corollary of Lemma \ref{obvious} and Corollary \ref{cor3}, we also obtain that, unlike in the class of all groups, there are no $SQ$-universal groups in the classes of elementary amenable groups and amenable groups.
\end{rem}

\paragraph{Elementary classes.} Recall that the class of amenable groups is closed under the following operations \cite{vN}:
\begin{enumerate}
\item[(S)] Taking subgroups.
\item[(Q)] Taking quotients.
\item[(E)] Taking extensions.
\item[(U)] Taking direct unions.
\end{enumerate}
The class of elementary amenable groups $EA$ was defined by Chou \cite{Cho} as the smallest class containing all finite and abelian groups and closed under the four operations (S)--(U). Alternatively, one can define $EG$ inductively as follows. Let $EG_0$ be the class of all finite and abelian groups. Let $\alpha $ be an ordinal. If $\alpha $ is limit, define $EG_\alpha =\bigcup_{\beta<\alpha} EG_\beta$. If $\alpha $ is a successor ordinal, let $EG_\alpha $ be the class of groups that can be obtained from groups in $EG_{\alpha -1}$ by applying (E) or (U) once. The following lemma summarizes some results of \cite{Cho}.

\begin{lem}[Chou]\label{Chou}
\begin{enumerate}
\item[(a)] For every ordinal $\alpha $, $EG_\alpha$ is closed under (S) and (Q).
\item[(b)] $EA=\bigcup_{\alpha} EG_\alpha$, where the union is taken over all ordinals.
\end{enumerate}
\end{lem}

Given a group $G\in EA$, define the \emph{elementary class} of $G$, $\c (G)$ as the smallest ordinal $\alpha $ such that $G\in EG_\alpha $. The following three observations are quite elementary.

\begin{lem}\label{succ}
For any group $G$, $\c (G)$ is a non-limit ordinal or $0$. If $G$ is countable, then $\c (G)<\omega _1$. If $G$ is finitely generated, then $\c (G)$ is $0$ or a successor of a non-limit ordinal.
\end{lem}

\begin{proof}
The first claim obviously follows from the definition of $EG_\alpha$ for a limit ordinal $\alpha$.

Further suppose that $G$ is countable and $\c (G)=\omega_1 +1$. Note that $G$ cannot split as an extension $$ 1\to N\to G\to Q\to 1,$$ where $\c (N)<\omega _1$ and $\c(Q)<\omega _1$. Indeed otherwise $N,Q\in EG_\beta$ for some $\beta <\omega _1 $ and hence $\c (G)\le \beta +1<\omega _1 $. Thus $G$ is a direct union of a family of groups $\mathcal H=\{ H_i\} _{i\in I}$ such that $\c(H_i)=\beta_i$ is countable. We enumerate the group $G=\{ 1, g_1, g_2, \ldots \}$ and choose a chain $\{ H_0, H_1, \ldots \}\subseteq \mathcal H$ as follows.  Let $H_0$ be any subgroup from $\mathcal H$. If $H_i$ is already chosen, let $H_{i+1}$ be any subgroup from $\mathcal H$ that contains $H_i$ and $g_i$. Obviously $G=\bigcup_{i=0}^\infty H_i$ and $\beta =\sup\limits_{i\in \mathbb N\cup\{0\}} \beta_i$ is countable. Consequently, $\c(G)\le \beta+1<\omega_1$ and we get a contradiction again.

Thus every group $G$ with $\c (G)=\omega_1 +1$ is uncountable. Using the same arguments as in the previous paragraph, it is now easy to show by transfinite induction that for every ordinal $\alpha \ge 1$, every group $G$ with $\c (G)=\omega_1 +\alpha $ is uncountable.

Finally suppose that $G$ is finitely generated and $\c (G)=\alpha +1$, where $\alpha $ is limit. As above we can show that $G$ cannot split as an extension of a group from $EG_{\alpha}$ by  a group from $EG_{\alpha}$. Hence $G$ is a direct union of groups $\{H_i\}_{i\in I}$ from $EG_\alpha $. Since $G$ is finitely generated, $G=H_i$ for some $i$ and hence $c(G)\le \alpha $. A contradiction.
\end{proof}

In what follows, we denote by $G^\omega $ the direct product of countably many copies of $G$.

\begin{prop}\label{EGprop}
Let $\alpha $ be a countable limit ordinal.
\begin{enumerate}
\item[(a)] There exists a countable group $H$ such that $\c (H)=\alpha +1$ and for every $K\in \mathcal E(H)$ we have $\c (K)\le \alpha +1$.
\item[(b)] For every $n\in \mathbb N$, $n\ge 2$, there exists a finitely generated group $L$ such that $\c (L)=\c (L^\omega)= \alpha +n$.
\end{enumerate}
\end{prop}

\begin{proof}
We first prove (a) by transfinite induction. Note that one can find ordinals $\beta_1, \beta_2, \ldots  <\alpha$ indexed by natural numbers such that $\sup_i \beta_i=\alpha$ and groups $H_1, H_2, \ldots $ such that $\c (H_i)=\beta _i$ for any $i\in \mathbb N$. Indeed this follows from the inductive hypothesis if $\alpha >\omega $. If $\alpha =\omega $, we can take $\beta_i=i$ and $H_i$ to be a non-cyclic free solvable groups of derived length $2^n$; it is easy to see  that $\c(H_n)=n$.

Let $H=\prod_{i=1}^\infty H_i$ be the direct product of $H_i$'s. Then $H$ is a direct union of groups $\prod_{i=1}^n H_i$ which all have classes less than $\alpha$. Hence $\c (H)\le \alpha +1$. On the other hand, by the first part of Lemma \ref{Chou}, we have $\c (H)\ge c(H_i)$ for every $i\in \mathbb N$ and therefore $\c(H)\ge \alpha $. Now Lemma \ref{succ} implies that $\c (H)=\alpha +1$.

Let us show that $\c(K)\le \alpha +1$ for every $K\in \mathcal E(H)$. As $K$ is the direct union of its finitely generated subgroups, it suffices to show that every finitely generated subgroup $S$ of a direct power of $H$ satisfies $\c (S)\le \alpha $. Since $S$ is finitely generated, it is a subgroup of a direct product of finitely many $H_i$'s. The latter product has elementary class less than $\alpha$ and hence $\c(S)< \alpha $ by the first part of Lemma \ref{Chou}. This completes the proof of (a).

To prove (b) we have to consider two cases. First assume that $n=2$. Let $H$ be the countable group of elementary class $\alpha +1$ provided by the first part of the proposition. Let $G$ be the finitely generated group containing $H$ from Corollary \ref{cor1}. Since $G_1\cap H=\{ 1\}$ (see Theorem 1.1 (b)), $H$ also embeds in $L=G/G_1$ and the latter group is an extension of $G_2/G_1\in \mathcal E(H)$ by a solvable group $G/G_2$ of derived length $3$. In particular, we have $\c (G_2/G_1)\le \alpha +1 $ and hence $\c (L)\le \alpha +2$. On the other hand, $\c (L)\ge \c(H)\ge \alpha +1$ and $\c(L)$ cannot equal $\alpha +1$ by Lemma \ref{succ} as $L$ is finitely generated. Thus $\c (L)=\alpha +2$. Further we note that $L^\omega $ also splits as an extension of a group from $\mathcal E(H)$ by a solvable group of derived length $3$ and hence $\c (L^\omega)=\alpha +2$.

Further suppose that $n\ge 3$. By induction, there exists a finitely generated group $L_0$ such that $\c (L_0)=\c (L_0^\omega)= \alpha +n-1$. Let $L= L_0 \wre L_0$. Obviously
\begin{equation}\label{cL}
\alpha +n-1= \c (L_0)\le \c(L)\le \alpha +n.
\end{equation}
We want to show that, in fact, $\c(L)=\alpha +n$. Indeed suppose that $\c (L)= \alpha + n-1$. Then $L$ can be obtained from groups of elementary class at most $\alpha +n-2$ by applying  (E) or (U) once.

Suppose first that $L$ splits as $$ 1\to N\to L\to Q\to 1,$$ where $N, Q\in EG_{\alpha +n-2}$. Then $Q$ cannot contain a subgroup isomorphic to $L_0$ and hence the intersection of $N$ with the active copy of the group $L_0$ in $L_0\wre L_0$ is nontrivial. Let $a$ be a non-trivial element from this intersection and let $B$ be the image of the canonical embedding of $L_0$ in the base subgroup of $L_0\wre L_0$. Since $N$ is normal in $L$, it contains the subgroup $$D=[a, B]=\langle aba^{-1}b^{-1} \mid b\in B\rangle .$$ Note that elements $ab_1a^{-1}$ and $b_2$ commute in $L_0\wre L_0$ for any $b_1,b_2\in B$. Hence the map $aba^{-1}b^{-1}\mapsto b$, $b\in B$, extends to a homomorphism $D\to L_0$. Obviously it is surjective. Therefore, by the first part of Lemma \ref{Chou} we have
$$
\c (N)\ge \c (D)\ge \c (L_0)=\alpha +n-1,
$$
which contradicts our assumption.
Similarly if $L$ is a direct union of groups $\{ M_i\} _{i\in I}$ from $EG_{\alpha +n-2}$, then $L=M_i$ for some $i\in I$ as $L$ is finitely generated. Hence $\c (L)\le \alpha +n-2$, which contradicts (\ref{cL}). Thus we get a contradiction in both cases and hence $\c (L)= \alpha + n$.

It remains to show that $\c (L^\omega )=\alpha +n$. However this is obvious since $L^\omega $ splits as an extension of a group isomorphic to $L_0^\omega  $ by a group isomorphic to $L_0^\omega  $, which implies $$\c (L^\omega )\le \c (L_0^\omega ) +1=\alpha +n.$$
\end{proof}

\begin{proof}[Proof of Corollary \ref{EG}]
The corollary follows from Lemma \ref{succ}, Proposition \ref{EGprop}, and the observation about solvable groups made at the beginning of the proof of Proposition \ref{EGprop}.
\end{proof}

\paragraph{Acknowledgments.} We would like to thank V. Mikaelian for explaining to us details of the P. Hall's construction \cite{Hall}, S. Thomas for stimulating discussions and questions, and G. Arzhantseva and Y. de Cornulier for corrections and remarks. We are also grateful to anonymous referees for useful comments and remarks. The research of the first author was supported by the NSF grant DMS-0700811. The second author was supported by the NSF grant DMS-1006345. Both authors were supported by the RFBR grant 11-01-00945.

\vspace{1cm}

\noindent \textbf{Alexander Olshanskii: } Department of Mathematics, Vanderbilt University, Nashville 37240, U.S.A., and Department of Mechanics and Mathematics, Moscow State University, Russia.\\
E-mail: \emph{alexander.olshanskiy@vanderbilt.edu}\\

\smallskip

\noindent \textbf{Denis Osin: } Department of Mathematics, Vanderbilt University, Nashville 37240, U.S.A.\\
E-mail: \emph{denis.osin@gmail.com}

\end{document}